\documentclass[12]{amsart}
\usepackage{amsmath,amssymb}
\usepackage{amsthm,color,enumerate,comment,centernot,enumitem,url,cite}
\usepackage{graphicx,relsize,bm}
\usepackage{mathtools}
\usepackage{array}

\makeatletter
\newcommand{\tpmod}[1]{{\@displayfalse\pmod{#1}}}
\makeatother

\newcommand{\ord}{\operatorname{ord}}

\newtheorem{thm}{Theorem}[section]
\newtheorem{lemma}[thm]{Lemma}

\newtheorem{prop}[thm]{Proposition}
\newtheorem{cor}[thm]{Corollary}

\theoremstyle{remark}

\theoremstyle{definition}
    \newtheorem{defn}[thm]{Definition}

\newtheorem{rem}[thm]{Remark}

\newcommand{\abs}[1]{\left|{#1}\right|}

\def\FF {{\mathcal F}}

\def\Z {{\mathbb Z}}

\def\NN {{\mathcal N}}

\def\Q {{\mathbb Q}}

\def\C {{\mathcal C}}

\def\D {{\mathcal D}}

\def\F {{\mathbb F}}
\def\D {{\mathcal D}}
\def\Z {{\mathbb Z}}
\def\Q {{\mathbb Q}}
\def\C {{\mathbb C}}

\def\CC {{\mathcal C}}

\makeatletter
\@namedef{subjclassname@2020}{%
  \textup{2020} Mathematics Subject Classification}
\makeatother

\def\red#1 {\textcolor{red}{#1 }}
\def\blue#1 {\textcolor{blue}{#1 }}

\numberwithin{equation}{section}

\def\Z {{\mathbb Z}}
\begin{document}

\title[Trinomials and $k$-Wall-Sun-Sun Primes]{A Connection Between the Monogenicity of Certain Power-Compositional Trinomials and $k$-Wall-Sun-Sun Primes}


\author{Lenny Jones}
\address{Professor Emeritus, Department of Mathematics, Shippensburg University, Shippensburg, Pennsylvania 17257, USA}
\email[Lenny~Jones]{doctorlennyjones@gmail.com}

\date{\today}

\begin{abstract}
We say that a monic polynomial $f(x)\in {\mathbb Z}[x]$ of degree $N$ is \emph{monogenic} if $f(x)$ is irreducible over ${\mathbb Q}$ and
\[\{1,\theta,\theta^2,\ldots, \theta^{N-1}\}\]
is a basis for the ring of integers of ${\mathbb Q}(\theta)$, where $f(\theta)=0$.

Let $k$ be a positive integer, and let $U_n:=U_n(k,-1)$ be the Lucas sequence $\{U_n\}_{n\ge 0}$ of the first kind defined by
\[U_0=0,\quad U_1=1\quad \mbox{and} \quad U_n=kU_{n-1}+U_{n-2} \quad \mbox{ for $n\ge 2$}.\] A \emph{$k$-Wall-Sun-Sun prime} is a prime $p$ such that
\[U_{\pi_k(p)}\equiv 0 \pmod{p^2},\]
where $\pi_k(p)$ is the length of the period of $\{U_n\}_{n\ge 0}$ modulo $p$.

 Let ${\mathcal D}=k^2+4$ if $k\equiv 1 \pmod{2}$, and ${\mathcal D}=(k/2)^2+1$ if $k\equiv 0 \pmod{2}$. Suppose that $k\not \equiv 0 \pmod{4}$ and ${\mathcal D}$ is squarefree, and let $h$ denote the class number of ${\mathbb Q}(\sqrt{{\mathcal D}})$. Let $s\ge 1$ be an integer such that, for every odd prime divisor $p$ of $s$,  ${\mathcal D}$ is not a square modulo $p$ and  $\gcd(p,h{\mathcal D})=1$. In this article, we prove that
  $x^{2s^n}-kx^{s^n}-1$ is monogenic for all integers $n\ge 1$ if and only if no prime divisor of $s$ is a $k$-Wall-Sun-Sun prime.
\end{abstract}


\maketitle
\section{Introduction}\label{Section:Intro}

For a positive integer $k$, we let $U_n:=U_{n}(k,-1)$ denote the $n$th term of the Lucas sequence $\{U_n\}_{n\ge 0}$ of the first kind defined by
\begin{equation}\label{Eq:Lucas}
U_0=0,\quad U_1=1\quad \mbox{and} \quad U_n=kU_{n-1}+U_{n-2} \quad \mbox{ for $n\ge 2$}.
\end{equation}
The sequence $\{U_n\}_{n\ge 0}$ is periodic modulo any prime $p$, and we let $\pi_k(p)$ denote the length of the period of $\{U_n\}_{n\ge 0}$ modulo $p$.

 A \emph{$k$-Wall-Sun-Sun prime} \cite{Wiki2} is a prime $p$ such that
 \begin{equation}\label{Def:kWSS}
 U_{\pi_k(p)}\equiv 0 \pmod{p^2}.
  \end{equation} When $k=1$, the sequence $\{U_n\}_{n\ge 0}$ is the well-known Fibonacci sequence, and the $k$-Wall-Sun-Sun primes in this case are also known as \emph{Fibonacci-Wieferich} primes \cite{Wiki1}, or simply \emph{Wall-Sun-Sun} primes \cite{CDP,Wiki2}. However, at the time this article was written, no such primes were known to exist. The existence of Wall-Sun-Sun primes was first investigated by D. D. Wall \cite{Wall} in 1960, and subsequently studied by the Sun brothers \cite{SunSun}, who showed a connection with Fermat's Last Theorem.

Throughout this article,  we let $\Delta(f)$ and $\Delta(K)$ denote, respectively, the discriminants over $\Q$ of $f(x)\in \Z[x]$ and a number field $K$. We define $f(x)\in \Z[x]$ to be \emph{monogenic} if $f(x)$ is monic, irreducible over $\Q$ and
  $\{1,\theta,\theta^2,\ldots ,\theta^{\deg(f)-1}\}$ is a basis for the ring of integers $\Z_K$ of $K=\Q(\theta)$, where $f(\theta)=0$.
 If $f(x)$ is irreducible over $\Q$ with $f(\theta)=0$,  then \cite{Cohen}
\begin{equation} \label{Eq:Dis-Dis}
\Delta(f)=\left[\Z_K:\Z[\theta]\right]^2\Delta(K).
\end{equation}
Observe then, from \eqref{Eq:Dis-Dis}, that $f(x)$ is monogenic if and only if $\Delta(f)=\Delta(K)$.
Thus, if $\Delta(f)$ is squarefree, then $f(x)$ is monogenic from \eqref{Eq:Dis-Dis}. However, the converse does not hold in general, and when $\Delta(f)$ is not squarefree, it can be quite difficult to determine whether $f(x)$ is monogenic.

Throughout this article, we also let $\{U_n\}_{n\ge 0}$ be the sequence as defined in \eqref{Eq:Lucas}, and let $\D$ be defined as follows:
\begin{defn}\label{Def:D}
  Let $k\ge 1$ be an integer, with $k\ne 4$. Define
 \[\D:=\left\{\begin{array}{cl}
  k^2+4 & \mbox{if $k\equiv 1 \pmod{2}$}\\
   (k/2)^2+1 & \mbox{if $k\equiv 0 \pmod{2}$.}
 \end{array} \right.\]
 \end{defn}
 \noindent
We let $\delta$ denote the Legendre symbol $\left(\frac{\D}{p}\right)$, where $p$ is a prime determined by the context.

In this article, we establish a connection between the monogenicity\footnote{Although the terms \emph{monogenity} and \emph{monogeneity} are more common in the literature, we have decided to use the more grammatically-correct term \emph{monogenicity}.} of certain power-compositional trinomials and $k$-Wall-Sun-Sun primes. More precisely, we   prove
\begin{thm}\label{Thm:Main}
Let $f(x)=x^2-kx-1\in \Z[x]$, such that $k\ge 1$, $k\not \equiv 0 \pmod{4}$ and $\D$ is squarefree. Let $s\ge 1$ be an integer such that $\delta=-1$ and  $\gcd(p,h\D)=1$ for every prime divisor $p\ge 3$ of $s$, where $h$ is the class number of $\Q(\sqrt{\D})$.
 Then $f(x^{s^n})$ is monogenic for all integers $n\ge 1$ if and only if no prime divisor of $s$ is a $k$-Wall-Sun-Sun prime.
\end{thm}

Theorem \ref{Thm:Main} is motivated in part by the recent result \cite[Theorem 4.5]{LJEisenstein} that if $f(x)=x^2+ax+a$ with $a\in \{2,3\}$, then  $f(x^{s^n})$ is monogenic for all integers $n\ge 0$ if and only if $s\ge 2$ has no prime divisors $p$ with the property that $a^{p-1}-1\equiv 0 \pmod{p^2}$ for some integer $a>1$. If $a^{p-1}-1\equiv 0 \pmod{p^2}$ for some integer $a>1$, then $p$ is called a \emph{base-$a$ Wieferich prime} \cite{Wiki1}.

A second motivation for this article arises from recent results of Bouazzaoui \cite{Bouazzaoui1,Bouazzaoui2}. Let $p\ge 3$ be a rational prime. A number field $K$ is said to be \emph{$p$-rational} if the Galois group
of the maximal pro-$p$-extension of $K$ which is unramified outside $p$ is a free pro-$p$-group of rank
$r_2 + 1$, where $r_2$ is the number of pairs of complex embeddings of $K$. Let $d>0$ be a fundamental discriminant \cite{Wiki0}, and let $h_d$ be the class number of the real quadratic field $K=\Q(\sqrt{d})$. For any unit $u\in K$ with $u\not \in \{\pm1\}$, Bouazzaoui defines a rational prime $p\ge 3$ to be \emph{Wieferich of basis $u$} if
\[u^{p^r-1}-1\equiv 0 \pmod{p^2},\]
where $r$ is the residual degree of $p$ in $K$. Let $\varepsilon$ be the fundamental unit of $K$, and let $\left(\frac{d}{p}\right)$ be the Legendre symbol. 
Bouazzaoui proves
\begin{thm}\label{Thm:B}
Suppose that $p\ge 3$ is a prime such that $p\nmid (\varepsilon-\overline{\varepsilon})^2h_d$. Then
\begin{align}\label{B}
\begin{split}
\mbox{$K$ is not $p$-rational} & \quad \Longleftrightarrow \quad \mbox{$p$ is Wieferich of basis $\varepsilon$}\\
 & \quad \Longleftrightarrow \quad \pi(p)=\pi(p^2),\\
 & \quad \Longleftrightarrow \quad F_{p-\left(\frac{d}{p}\right)} \equiv 0 \pmod{p^2},
 \end{split}
 \end{align}
 where $\{F_n\}_{n\ge 0}$ is the Lucas sequence of the first kind defined by
\[F_0=0,\quad F_1=1\quad \mbox{and} \quad F_n=(\varepsilon+\overline{\varepsilon})U_{n-1}-\NN_{K/\Q}(\varepsilon)U_{n-2} \quad \mbox{ for $n\ge 2$},\] with $\pi(p)$ and $\pi(p^2)$ the respective period lengths of $\{F_n\}_{n\ge 0}$ modulo $p$ and $p^2$.
  \end{thm}

\begin{rem}
  Theorem \ref{Thm:B} generalizes a theorem of Greenberg \cite{Greenberg}.
  \end{rem}

\section{Preliminaries}\label{Section:Prelim}
The formula for the discriminant of an arbitrary monic trinomial, due to Swan \cite{Swan}, is given in the following theorem.
\begin{thm}
\label{Thm:Swan}
Let $f(x)=x^N+Ax^M+B\in \Z[x]$, where $0<M<N$. Let $r=\gcd(N,M)$, $N_1=N/r$ and $M_1=M/r$. Then
\begin{equation*}\label{Eq:Del(f)}
\Delta(f)=(-1)^{N(N-1)/2}B^{M-1}D^r,
\end{equation*} where
\begin{equation}\label{Eq:D}
D:=N^{N_1}B^{N_1-M_1}-(-1)^{N_1}M^{M_1}(N-M)^{N_1-M_1}A^{N_1}.
\end{equation}
\end{thm}

The next two theorems are due to Capelli \cite{S}.
 \begin{thm}\label{Thm:Capelli1}  Let $f(x)$ and $h(x)$ be polynomials in $\Q[x]$ with $f(x)$ irreducible. Suppose that $f(\alpha)=0$. Then $f(h(x))$ is reducible over $\Q$ if and only if $h(x)-\alpha$ is reducible over $\Q(\alpha)$.
 \end{thm}

\begin{thm}\label{Thm:Capelli2}  Let $c\in \Z$ with $c\geq 2$, and let $\alpha\in\C$ be algebraic.  Then $x^c-\alpha$ is reducible over $\Q(\alpha)$ if and only if either there is a prime $p$ dividing $c$ such that $\alpha=\beta^p$ for some $\beta\in\Q(\alpha)$ or $4\mid c$ and $\alpha=-4\beta^4$ for some $\beta\in\Q(\alpha)$.
\end{thm}

The following theorem is a compilation of results from various sources.
\begin{thm}\label{Thm:Period}
  Let $p$ be a prime.
  \begin{enumerate}
  \item \label{I-1} $\pi_k(p)=2$ if and only if $k\equiv 0 \pmod{p}$.
   \item \label{I0:p=2} If $p=2$, then \[\pi_k(p)=\left\{\begin{array}{cl}
    2 & \mbox{if $k\equiv 0 \pmod{2}$}\\
    3 & \mbox{if $k\equiv 1 \pmod{2}$}
  \end{array}\right. \quad\mbox{and} \quad \pi_k(p^2)=\left\{\begin{array}{cl}
    2 & \mbox{if $k\equiv 0 \pmod{4}$}\\
    6 & \mbox{if $k\equiv 1 \pmod{4}$}\\
    4 & \mbox{if $k\equiv 2 \pmod{4}$}\\
    6 & \mbox{if $k\equiv 3 \pmod{2}$.}
  \end{array}\right.\]
   \item \label{I:even} If $p\ge 3$, then $\pi_k(p)\equiv  0\pmod{2}$.
    \item  \label{I4:R} If $p\ge 3$, then $\pi_k(p^2)\in \{\pi_k(p),p\pi_k(p)\}$.
    \item  \label{I1:QR} If $p\ge 3$ and $\delta=1$, then $p-1\equiv 0 \pmod{\pi_k(p)}$.
    \item  \label{I2:QNR} If $\delta=-1$, then $2(p+1)\equiv 0 \pmod{\pi_k(p)}$.
     \end{enumerate}
\end{thm}
\begin{proof}
Items \eqref{I-1} and \eqref{I0:p=2} of Theorem \ref{Thm:Period} follow easily by direct calculation, item \eqref{I:even} is a special case of work found in \cite{FP}, item \eqref{I4:R} is a special case of a result in \cite{Renault}, while items \eqref{I1:QR} and \eqref{I2:QNR} follow from
two theorems in \cite{GRS}.
\end{proof}

The next proposition appears as Proposition 1 in \cite{Y}.
\begin{prop}\label{Prop:Yokoi}
Let $k\ge 1$ be an integer, such that $k\ne 4$ and $\D$ is squarefree. Then $\varepsilon:=(k+\sqrt{k^2+4})/2$ is the fundamental unit of $\Q(\sqrt{\D})$ with $\NN(\varepsilon)=-1$, where $\NN:=\NN_{\Q(\varepsilon)/\Q}$ denotes the algebraic norm.
\end{prop}
Whenever the hypotheses of Proposition \ref{Prop:Yokoi} hold in the sequel, we assume that $\varepsilon$ denotes the fundamental unit $(k+\sqrt{k^2+4})/2$ of $\Q(\sqrt{\D})$. 

The following theorem, known as \emph{Dedekind's Index Criterion}, or simply \emph{Dedekind's Criterion} if the context is clear, is a standard tool used in determining the monogenicity of a polynomial.
\begin{thm}[Dedekind \cite{Cohen}]\label{Thm:Dedekind}
Let $K=\Q(\theta)$ be a number field, $T(x)\in \Z[x]$ the monic minimal polynomial of $\theta$, and $\Z_K$ the ring of integers of $K$. Let $p$ be a prime number and let $\overline{ * }$ denote reduction of $*$ modulo $p$ (in $\Z$, $\Z[x]$ or $\Z[\theta]$). Let
\[\overline{T}(x)=\prod_{i=1}^k\overline{\tau_i}(x)^{e_i}\]
be the factorization of $T(x)$ modulo $p$ in $\F_p[x]$, and set
\[g(x)=\prod_{i=1}^k\tau_i(x),\]
where the $\tau_i(x)\in \Z[x]$ are arbitrary monic lifts of the $\overline{\tau_i}(x)$. Let $h(x)\in \Z[x]$ be a monic lift of $\overline{T}(x)/\overline{g}(x)$ and set
\[F(x)=\dfrac{g(x)h(x)-T(x)}{p}\in \Z[x].\]
Then
\[\left[\Z_K:\Z[\theta]\right]\not \equiv 0 \pmod{p} \Longleftrightarrow \gcd\left(\overline{F},\overline{g},\overline{h}\right)=1 \mbox{ in } \F_p[x].\]
\end{thm}

The next result is essentially an algorithmic adaptation of Theorem \ref{Thm:Dedekind} specifically for trinomials. 
\begin{thm}{\rm \cite{JKS2}}\label{Thm:JKS}
Let $N\ge 2$ be an integer.
Let $K=\Q(\theta)$ be an algebraic number field with $\theta\in \Z_K$, the ring of integers of $K$, having minimal polynomial $f(x)=x^{N}+Ax^M+B$ over $\Q$, with $\gcd(M,N)=r$, $N_1=N/r$ and $M_1=M/r$. Let $D$ be as defined in \eqref{Eq:D}. A prime factor $p$ of $\Delta(f)$ does not divide $\left[\Z_K:\Z[\theta]\right]$ if and only if $p$ satisfies one of the following items: 
\begin{enumerate}[font=\normalfont]
  \item \label{JKS:I1} when $p\mid A$ and $p\mid B$, then $p^2\nmid B$;
  \item \label{JKS:I2} when $p\mid A$ and $p\nmid B$, then
  \[\mbox{either } \quad p\mid A_2 \mbox{ and } p\nmid B_1 \quad \mbox{ or } \quad p\nmid A_2\left((-B)^{M_1}A_2^{N_1}-\left(-B_1\right)^{N_1}\right),\]
  where $A_2=A/p$ and $B_1=\frac{B+(-B)^{p^e}}{p}$ with $p^e\mid\mid N$;
  \item \label{JKS:I3} when $p\nmid A$ and $p\mid B$, then
  \[\mbox{either } \quad p\mid A_1 \mbox{ and } p\nmid B_2 \quad \mbox{ or } \quad p\nmid A_1B_2^{M-1}\left((-A)^{M_1}A_1^{N_1-M_1}-\left(-B_2\right)^{N_1-M_1
  }\right),\]
  where $A_1=\frac{A+(-A)^{p^j}}{p}$ with $p^j\mid\mid (N-M)$, and $B_2=B/p$;
  \item \label{JKS:I4} when $p\nmid AB$ and $p\mid M$ with $N=up^m$, $M=vp^m$, $p\nmid \gcd\left(u,v\right)$, then the polynomials
   \begin{equation*}
     x^{N/p^m}+Ax^{M/p^m}+B \quad \mbox{and}\quad \dfrac{Ax^{M}+B+\left(-Ax^{M/p^m}-B\right)^{p^m}}{p}
   \end{equation*}
   are coprime modulo $p$;
         \item \label{JKS:I5} when $p\nmid ABM$, then $p^2\nmid D/r^{N_1}$.
   \end{enumerate}
\end{thm}
\begin{rem}
We will find both Theorem \ref{Thm:Dedekind} and Theorem \ref{Thm:JKS} useful in our investigations.
\end{rem}

The next theorem follows from Corollary (2.10) in \cite{Neukirch}.
\begin{thm}\label{Thm:CD}
  Let $K$ and $L$ be number fields with $K\subset L$. Then \[\Delta(K)^{[L:K]} \bigm|\Delta(L).\]
\end{thm}

\section{The Proof of Theorem \ref{Thm:Main}}

We first prove some lemmas.
\begin{lemma}\label{Lem:Irreducible}
Let $k$ and $s$ be positive integers with $k\ne 4$. Let $f(x)=x^2-kx-1$.  If $\D$ is squarefree, then $f(x^{s^n})$ is irreducible over $\Q$ for all integers $n\ge 1$.
\end{lemma}
\begin{proof}
Since $\D>1$ is squarefree, it follows that $f(x)$ is irreducible over $\Q$, and the trivial case of $s=1$ is true. Suppose then that $s\ge 2$.
   Note that $f(\varepsilon)=0$.   Let $h(x)=x^{s^n}$ and assume, by way of contradiction, that $f(h(x))$ is reducible. Then, by Theorems \ref{Thm:Capelli1} and \ref{Thm:Capelli2} (with $\alpha=\varepsilon$), we have, for some $\beta\in \Q(\varepsilon)$, that either $\varepsilon=\beta^p$ for some prime $p$ dividing $s$, or $\varepsilon=-4\beta^4$ if $s^n\equiv 0 \pmod{4}$. Thus, it is immediate that $\varepsilon=-4\beta^4$ is impossible since $\NN(\varepsilon)=-1$ and $\NN(-4\beta)\equiv 0 \pmod{16}$. Hence, $\varepsilon=\beta^p$ for some prime divisor $p$ of $s$. Then, we see by taking norms that
  \[-1=\NN(\varepsilon)=\NN(\beta)^p,\] which implies that $p\equiv 1 \pmod{2}$ and $\NN(\beta)=-1$, since $\NN(\beta)\in \Z$. Thus, $\beta$ is a unit, and therefore $\beta=\pm \varepsilon^j$ for some $j\in \Z$, since $\varepsilon$ is the fundamental unit of $\Q(\sqrt{\D})$ by Proposition \ref{Prop:Yokoi}. 
   Consequently,
  \[\varepsilon=\beta^p=(\pm 1)^p\varepsilon^{jp},\] which implies that $(\pm 1)^p\varepsilon^{jp-1}=1$, contradicting the fact that $\varepsilon$ has infinite order in the unit group of the ring of algebraic integers of the real quadratic field $\Q(\sqrt{\D})$.
\end{proof}
\begin{lemma}\label{Lem:Basic1}
  Let $k\ge 1$ be an integer, and let $f(x)=x^2-kx-1$. Then $f(x)$ is monogenic if and only if $k\not \equiv 0 \pmod{4}$ and $\D$ is squarefree.
\end{lemma}
\begin{proof}
By Lemma \ref{Lem:Irreducible}, $f(x)$ is irreducible over $\Q$. Let $f(\theta)=0$, and let $p$ be a prime divisor of $\Delta(f)=k^2+4$. To examine the monogenicity of $f(x)$, we use Theorem \ref{Thm:JKS}. Suppose first that $p\mid k$. Then $p=2$, and item \eqref{JKS:I2} of Theorem \ref{Thm:JKS} applies. Since $A_2=k/2$ and $B_1=0$, it is easy to see that
 \[\left[\Z_K:\Z[\theta]\right]\not \equiv 0 \pmod{2} \quad\Longleftrightarrow \quad k \not \equiv 0 \pmod{4}.\]
  Suppose next that $p\nmid k$, so that $p\ne 2$. Then, by item \eqref{JKS:I5} of Theorem \ref{Thm:JKS}, we deduce that
 \[\left[\Z_K:\Z[\theta]\right]\not \equiv 0 \pmod{p} \quad\Longleftrightarrow \quad k^2+4 \not \equiv 0 \pmod{p^2},\] which completes the proof.
 \end{proof}

\begin{lemma}\label{Lem:k not squarefree}
Let $k\in \Z$ with $k\ge 1$, and let $p$ be a prime.
\begin{enumerate}
\item \label{I1: p=2} If $p=2$, then $p$ is a $k$-Wall-Sun-Sun prime if and only if $k\equiv 0 \pmod{4}$.
\item \label{I2: p>2} If $p\ge 3$ and $k\equiv 0 \pmod{p^2}$, then $p$ is a $k$-Wall-Sun-Sun prime.
\end{enumerate}
\end{lemma}
\begin{proof}
  Item \eqref{I1: p=2} is easily verified by direct calculation using Theorem \ref{Thm:Period}. For item \eqref{I2: p>2}, note that $U_2=k$ and $U_3=k^2+1$. Thus, since $k\equiv 0 \pmod{p^2}$, we have that $\pi_k(p)=2$ and $U_2\equiv 0 \pmod{p^2}$.
 \end{proof}

For the next lemma, we let $\ord_m(*)$ denote the order of $*$ modulo the integer $m\ge 2$.
\begin{lemma}\label{Lem:Order}
Let $k\ge 1$ be an integer, such that $k\ne 4$ and $\D$ is squarefree. Let $p\ge 3$ be a prime such that $\delta=-1$.
     Then
     \begin{enumerate}
     \item \label{Per I:1} $\ord_m(\varepsilon)=\ord_m(\overline{\varepsilon})=\pi_k(m)$ for $m\in \{p,p^2\}$, 
     \item \label{Per I:2} $\varepsilon^{p+1}+1\equiv 0\pmod{p}$.
     \end{enumerate}
\end{lemma}
\begin{proof}
  It follows from \cite{Robinson} that the order modulo an odd integer $m\ge 3$ of the companion matrix $\CC$ for the characteristic polynomial of $\{U_n\}_{n\ge 0}$ is $\pi_k(m)$. The characteristic polynomial of $\{U_n\}_{n\ge 0}$ is $f(x)=x^2-kx-1$, so that
  \[\CC=\left[\begin{array}{cc}
    0&1\\
    1&k
  \end{array}\right].\] Since the eigenvalues of $\CC$ are $\varepsilon$ and $\overline{\varepsilon}$, we conclude that
  \[\ord_m\left(\left[\begin{array}{cc}
    \varepsilon&0\\
    0&\overline{\varepsilon}
  \end{array}\right]\right)=\ord_m(\CC)=\pi_k(m), \quad \mbox{for $m\in \{p,p^2\}$.}
  \] Let $z\ge 1$ be an integer, and suppose that $\varepsilon^z=a+b\sqrt{\D}\in \Q(\sqrt{\D})$. Then
    $\NN(\varepsilon^z)=a^2-\D b^2$. But $\NN(\varepsilon^z)=\NN(\varepsilon)^z=(-1)^z$, so that $a^2-\D b^2=(-1)^z$. Thus,
  \[\overline{\varepsilon}^z=\left(-1/\varepsilon\right)^z=(-1)^z/(a+b\sqrt{\D})=(-1)^z(a-b\sqrt{\D})/(a^2-\D b^2)=a-b\sqrt{\D}.\]
  Hence, since $\delta=-1$, it follows that
  \[\varepsilon^z\equiv 1 \pmod{m} \quad \mbox{if and only if} \quad \overline{\varepsilon}^z\equiv 1 \pmod{m}\]
  for $m\in \{p,p^2\}$,
  which establishes item \eqref{Per I:1}.

  For item \eqref{Per I:2}, since $\delta=-1$,
  we have by Euler's criterion that
   \[\left(\sqrt{k^2+4}\right)^{p+1}=(k^2+4)^{(p-1)/2}(k^2+4)\equiv \delta(k^2+4)\equiv -(k^2+4) \pmod{p},\] which implies 
   \[\left(\sqrt{k^2+4}\right)^{p}\equiv -\sqrt{k^2+4} \pmod{p}.\]
   Hence,
   \begin{align*}\label{Eq:Expansion}
     \varepsilon^{p+1}&=\left(\frac{k+\sqrt{k^2+4}}{2}\right) \left(\frac{k+\sqrt{k^2+4}}{2}\right)^{p}\\
     &=\left(\frac{k+\sqrt{k^2+4}}{2}\right) \sum_{j=0}^p\binom{p}{j}\left(\frac{k}{2}\right)^j\left(\frac{\sqrt{k^2+4}}{2}\right)^{p-j}\\
     &\equiv \left(\frac{k+\sqrt{k^2+4}}{2}\right)\left(\left(\frac{k}{2}\right)^p+\left(\frac{\sqrt{k^2+4}}{2}\right)^{p}\right) \pmod{p}\\
     &\equiv \left(\frac{k+\sqrt{k^2+4}}{2}\right)\left(\frac{k-\sqrt{k^2+4}}{2}\right) \pmod{p}\\
     &\equiv -1 \pmod{p},
    \end{align*}
   which completes the proof of the lemma.
 \end{proof}

\begin{lemma}\label{Lem:NewEq}
Let $p\ge 3$ be a prime. Then $p$ is a $k$-Wall-Sun-Sun prime if and only if $U_{p-\delta}\equiv 0 \pmod{p^2}$.
\end{lemma}
\begin{proof} We provide details only for $\delta=-1$ since the proof is similar when $\delta=1$.

  Suppose first that $p$ is a $k$-Wall-Sun-Sun prime. Then
  \begin{equation}\label{Eq:kWSS}
  U_{\pi_k(p)}\equiv 0\pmod{p^2}.
  \end{equation} Using the Binet-formula representation and item \eqref{I:even} of Theorem \ref{Thm:Period}, we have that
   \begin{equation}\label{Eq:Upi}
  U_{\pi_k(p)}=\frac{\varepsilon^{\pi_k(p)}-(-1/\varepsilon)^{\pi_k(p)}}{\varepsilon+1/\varepsilon}
  =\frac{(\varepsilon^{\pi_k(p)}-1)(\varepsilon^{\pi_k(p)}+1)}{\varepsilon^{\pi_k(p)}(\varepsilon+1/\varepsilon)}.
  \end{equation}
   By Lemma \ref{Lem:Order},  $\varepsilon^{\pi_k(p)}\equiv 1 \pmod{p}$, so that $\varepsilon^{\pi_k(p)}+1\equiv 2\pmod{p}$. Hence,
  \[\varepsilon^{\pi_k(p)}-1\equiv 0 \pmod{p^2}\] by \eqref{Eq:kWSS}. Thus,
  if $\delta=-1$, then
  \begin{equation}\label{Eq:Udelta}
  U_{p-\delta}=\frac{\varepsilon^{2(p+1)-1}-1}{\varepsilon^{p+1}(\varepsilon+1/\varepsilon)}\equiv 0 \pmod{p^2},
  \end{equation}
  by item \eqref{I2:QNR} of Theorem \ref{Thm:Period}, which completes the proof in this direction.

  Conversely, with $\delta=-1$, suppose that \eqref{Eq:Udelta} holds.
     From item \eqref{I2:QNR} of Theorem \ref{Thm:Period}, we can write $2(p+1)=z\pi_k(p)$. Then
  \[\varepsilon^{2(p+1)}-1=(\varepsilon^{\pi_k(p)}-1)S \equiv 0\pmod{p^2},\]
  where
  \[S=(\varepsilon^{\pi_k(p)})^{z-1}+(\varepsilon^{\pi_k(p)})^{z-2}+\cdots +\varepsilon^{\pi_k(p)}+1\equiv z\not \equiv 0\pmod{p},\] since $\varepsilon^{\pi_k(p)}\equiv 1\pmod{p}$ by Lemma \ref{Lem:Order}. Thus, $\varepsilon^{\pi_k(p)}-1\equiv 0\pmod{p^2}$, which implies that $U_{\pi_k(p)}\equiv 0 \pmod{p^2}$ by \eqref{Eq:Upi}, completing the proof of the lemma.
\end{proof}
  Note that $\D$ and $4\D$ are fundamental discriminants when $k\equiv 1 \pmod{2}$ and $k\equiv 0 \pmod{2}$, respectively. The next lemma then follows from Theorem \ref{Thm:B} and Lemma \ref{Lem:NewEq}.
\begin{lemma}\label{Lem:Equivalent Conditions}
Let $k\in \Z$ with $k\ge 1$, and let $p$ be a prime such that $p\ge 3$ and $\gcd(\D h,p)=1$, where $h$ is the class number of $\Q(\sqrt{\D})$. Then
\begin{align*}
\mbox{$p$ is a $k$-Wall-Sun-Sun prime} & \quad \Longleftrightarrow \quad \varepsilon^{p^r-1}-1\equiv 0 \pmod{p^2},\\ 
 & \quad \Longleftrightarrow \quad \pi_k(p)=\pi_k(p^2),\\
 & \quad \Longleftrightarrow \quad U_{p-\delta}\equiv 0 \pmod{p^2},
 \end{align*}
 where 
 $r$ is the residual degree of $p$ in $\Q(\sqrt{\D})$.
\end{lemma}

\begin{lemma}\label{Lem:Main1}
  Let $k\in \Z$, such that $k\ge 1$, $k\not\equiv 0\pmod{4}$ and $\D$ is squarefree. Let $p$ be a prime such that $p\ge 3$, $p\nmid k$
  and $\gcd(p,\D h)=1$,  where $h$ is the class number of $\Q(\sqrt{\D})$.  If $\delta=-1$, then the following conditions are equivalent:
  \begin{enumerate}
    \item \label{Lem:Main1 I1} $p$ is a $k$-Wall-Sun-Sun prime,
    \item \label{Lem:Main1 I2} $\varepsilon^{2p^m}-k\varepsilon^{p^m}-1 \equiv 0\pmod{p^2}$ for all integers $m\ge 1$,
    \item \label{Lem:Main1 I3} $\varepsilon^{2p^m}-k\varepsilon^{p^m}-1 \equiv 0\pmod{p^2}$ for some integer $m\ge 1$.
  \end{enumerate}
  \end{lemma}
\begin{proof}
First, observe that item \eqref{Lem:Main1 I2} clearly implies item \eqref{Lem:Main1 I3}.

  We show next that item \eqref{Lem:Main1 I1} implies item \eqref{Lem:Main1 I2}.
       Since $\delta=-1$, we see from Theorem \ref{Thm:Period} that $2(p+1)\equiv 0 \pmod{\pi_k(p)}$. Thus, since $p$ is a $k$-Wall-Sun-Sun prime, it follows from Lemma \ref{Lem:Equivalent Conditions} that $\pi_k(p)=\pi_k(p^2)$. Hence, $2(p+1)\equiv 0 \pmod{\pi_k(p^2)}$ and
   \begin{equation}\label{Eq:power}
   \varepsilon^{2(p+1)}-1\equiv (\varepsilon^{p+1}-1)(\varepsilon^{p+1}+1)\equiv 0 \pmod{p^2}
   \end{equation} by part \eqref{Per I:1} of Lemma \ref{Lem:Order}.
    Since
    \begin{equation}\label{Eq:gcd}
    \gcd(\varepsilon^{p+1}-1,\varepsilon^{p+1}+1)\not \equiv 0 \pmod{p},
     \end{equation} we conclude from \eqref{Eq:power} and part \eqref{Per I:2} of Lemma \ref{Lem:Order} that
       \begin{equation}\label{Eq:Power1}
   \varepsilon^{p}\equiv -\varepsilon^{-1} \pmod{p^2}.
   \end{equation}
    Let $m\ge 1$ be an integer. Since
   \[p^m\equiv \left\{\begin{array}{cl}
   1 \pmod{2(p+1)}& \mbox{if $m\equiv 0 \pmod{2}$,}\\
   p \pmod{2(p+1)}& \mbox{if $m\equiv 1 \pmod{2}$,}
   \end{array}\right.\]
    we have from item \eqref{Per I:1} of Lemma \ref{Lem:Order} that 
   \[\varepsilon^{2p^m}-k\varepsilon^{p^m}-1\equiv \left\{\begin{array}{cl}
   \varepsilon^2-k\varepsilon-1 \pmod{p^2}& \mbox{if $m\equiv 0 \pmod{2}$,}\\
   \varepsilon^{2p}-k\varepsilon^{p}-1 \pmod{p^2}& \mbox{if $m\equiv 1 \pmod{2}$.}
   \end{array}\right.
   \]
   Using \eqref{Eq:Power1}, we deduce that
   \[\varepsilon^{2p}-k\varepsilon^{p}-1\equiv
   \dfrac{-(\varepsilon^2-k\varepsilon-1)}{\varepsilon^2} \pmod{p^2},\]
      which completes the proof in this direction since $\varepsilon^2-k\varepsilon-1=0$.

   Finally, we show that item \eqref{Lem:Main1 I3} implies  item \eqref{Lem:Main1 I1}.
   By items \eqref{I2:QNR} and \eqref{I4:R} of Theorem \ref{Thm:Period}, we deduce that $2p(p+1)\equiv 0 \pmod{\pi_k(p^2)}$. Let $m\ge 1$ be any integer. Then,
   since
   \[p^m\equiv \left\{\begin{array}{cl}
   p \pmod{2p(p+1)}& \mbox{if $m\equiv 1 \pmod{2}$,}\\
   p^2 \pmod{2p(p+1)}& \mbox{if $m\equiv 0 \pmod{2}$,}
   \end{array}\right.\]
    it follows from item \eqref{Per I:1} of Lemma \ref{Lem:Order} that
   \[\varepsilon^{2p^m}-k\varepsilon^{p^m}-1\equiv \left\{\begin{array}{cl}
   \varepsilon^{2p}-k\varepsilon^p-1 \pmod{p^2}& \mbox{if $m\equiv 1 \pmod{2}$,}\\
   \varepsilon^{2p^2}-k\varepsilon^{p^2}-1 \pmod{p^2}& \mbox{if $m\equiv 0 \pmod{2}$.}
   \end{array}\right.
   \] By assumption, we have that 
   \begin{equation}\label{Eq:2Possibilities}
   \mbox{either} \quad  \varepsilon^{2p}-k\varepsilon^p-1 \equiv 0 \pmod{p^2} \quad \mbox{or} \quad \varepsilon^{2p^2}-k\varepsilon^{p^2}-1 \equiv 0 \pmod{p^2}.
   \end{equation}

 The zeros of $f(x)=x^2-kx-1$ in the ring $(\Z/p\Z)[\D]$ are $\varepsilon$ and $\overline{\varepsilon}=-1/\varepsilon$. By Hensel, these zeros lift to the zeros $\varepsilon$ and $\overline{\varepsilon}$ of $f(x)$ in $(\Z/p^2\Z)[\D]$.
   Hence, if the first possibility of \eqref{Eq:2Possibilities} is true, then
   \begin{equation}\label{Eq:Poss}
   \varepsilon^p\equiv \varepsilon \pmod{p^2}\quad \mbox{or}\quad \varepsilon^p\equiv -1/\varepsilon \pmod{p^2}.
   \end{equation}
   Suppose the first possibility of \eqref{Eq:Poss} is true. Then, by item \eqref{Per I:2} of Lemma \ref{Lem:Order}, we have that
     \[-1\equiv \varepsilon^{p+1}\equiv \varepsilon^p\varepsilon\equiv \varepsilon^2 \equiv \frac{k^2+2+k\sqrt{k^2+4}}{2} \pmod{p}.\] Thus, \[\frac{k^2+4+k\sqrt{k^2+4}}{2}\equiv 0 \pmod{p},\] which implies that $k^2+4\equiv 0 \pmod{p}$, contradicting the fact that $\gcd(p,\D)=1$.  Thus, the second possibility of \eqref{Eq:Poss} holds, which implies that $\varepsilon^{p+1}\equiv -1 \pmod{p^2}$. Since $\delta=-1$, the residual degree of $p$ is $r=2$. Hence,
   \[\varepsilon^{p^r-1}-1\equiv \varepsilon^{p^2-1}-1\equiv \left(\varepsilon^{p+1}\right)^{p-1}-1\equiv (-1)^{p-1}-1\equiv 0 \pmod{p^2},\] which implies that $p$ is a $k$-Wall-Sun-Sun prime by Lemma \ref{Lem:Equivalent Conditions}.

    Suppose now that the second possibility,
   \begin{equation}\label{Eq:Zeros2}
   \varepsilon^{2p^2}-k\varepsilon^{p^2}-1 \equiv 0 \pmod{p^2},
   \end{equation}
   of \eqref{Eq:2Possibilities} is true. Then
   \begin{equation}\label{Eq:Poss2}
   \varepsilon^{p^2}\equiv \varepsilon \pmod{p^2}\quad \mbox{or}\quad \varepsilon^{p^2}\equiv -1/\varepsilon \pmod{p^2}.
   \end{equation}
      If the second possibility in \eqref{Eq:Poss2} holds, then $\varepsilon^{2(p^2+1)}\equiv 1 \pmod{p}$. Since $\ord_p(\varepsilon)=\pi_k(p)$ by item       \eqref{Per I:1} of Lemma \ref{Lem:Order}, it follows that $2p^2+2\equiv 0 \pmod{\pi_k(p)}$. By item \eqref{I2:QNR} of Theorem \ref{Thm:Period}, we have that $2p+2\equiv 0\pmod{\pi_k(p)}$. Hence,
      \[\pi_k(p) \quad \mbox{divides}\quad (2p^2+2)-(p-1)(2p+2)=4,\] so that $\pi_k(p)\in \{2,4\}$. Recall that $p\nmid k$ by hypothesis. Thus, $\pi_k(p)\ne 2$ by item \eqref{I-1} of Theorem \ref{Thm:Period}. If $\pi_k(p)=4$, then
      \begin{equation*}\label{Eq:p divisiblity}
      U_4=k(k^2+2)\equiv 0 \pmod{p}\quad \mbox{and}\quad U_5=k^2(k^2+3)+1\equiv 1\pmod{p}.
      \end{equation*}
      Hence, it follows that
      \[p\quad \mbox{divides}\quad (k^2+3)-(k^2+2)=1,\]
      which is impossible. Therefore, the first possibility in \eqref{Eq:Poss2} holds, which implies that $\varepsilon^{p^2-1}-1\equiv 0 \pmod{p^2}$, and the proof is complete by Lemma \ref{Lem:Equivalent Conditions}.
   \end{proof}
\begin{cor}\label{Cor:Main1}
   Let $k\in \Z$, such that $k\ge 1$, $k\not\equiv 0\pmod{4}$ and $\D$ is squarefree. Let $p$ be a prime such that $p\ge 3$, $p\nmid k$
  and $\gcd(p,\D h)=1$,  where $h$ is the class number of $\Q(\sqrt{\D})$.  If $\delta=-1$, then the following conditions are equivalent:
  \begin{enumerate}
    \item \label{Cor:Main1 I1} $p$ is a $k$-Wall-Sun-Sun prime,
    \item \label{Cor:Main1 I2} $\overline{\varepsilon}^{2p^m}-k\overline{\varepsilon}^{p^m}-1 \equiv 0\pmod{p^2}$ for all integers $m\ge 1$,
    \item \label{Cor:Main1 I3} $\overline{\varepsilon}^{2p^m}-k\overline{\varepsilon}^{p^m}-1 \equiv 0\pmod{p^2}$ for some integer $m\ge 1$.
  \end{enumerate}
\end{cor}
\begin{proof}
  Since $\overline{\varepsilon}=-1/\varepsilon$, we have that
  \[\overline{\varepsilon}^{2p^m}-k\overline{\varepsilon}^{p^m}-1=-\dfrac{\varepsilon^{2p^m}-k\varepsilon^{p^m}-1}{\varepsilon^{2p^m}}.\] Thus, the corollary follows from Lemma \ref{Lem:Main1}.
\end{proof}

 \begin{proof}[Proof of Theorem \ref{Thm:Main}]
 For brevity of notation, define
 \[\FF_n(x):=f(x^{s^n})=x^{2s^n}-kx^{s^n}-1.\] Since $\D$ is squarefree and $k\ne 4$, we have that $\FF_n(x)$ is irreducible for all $n\ge 1$ by Lemma \ref{Lem:Irreducible}.
 The case $s=1$ follows from Lemma \ref{Lem:Basic1}. So assume that $s\ge 2$.

  ($\Rightarrow$) Suppose that $s$ has a prime divisor $p$ that is a $k$-Wall-Sun-Sun prime. We claim that $\FF_1(x)$ is not monogenic. Let $\FF_1(\theta)=0$,  $K=\Q(\theta)$ and $\Z_K$ be the ring of integers of $K$. Since $k\not \equiv 0 \pmod{4}$, it follows from item \eqref{I1: p=2} of Lemma \ref{Lem:k not squarefree} that $p\ge 3$.  If $p\mid k$, then $\pi_k(p)=2$, so that $U_2=k\equiv 0 \pmod{p^2}$, since $p$ is a $k$-Wall-Sun-Sun prime. Thus, applying item \eqref{JKS:I2} of Theorem \ref{Thm:JKS} to $\FF_1(x)$, we see that
  \[B_1\equiv 0 \pmod{p} \quad \mbox{and} \quad A_2=k/p \equiv 0 \pmod{p},\] from which we conclude that $\left[\Z_K:\Z[\theta]\right]\equiv 0\pmod{p}$ and $\FF_1(x)$ is not monogenic.

  Suppose next that $p\nmid k$ and $p^m\mid \mid s$, with $m\ge 1$. We apply Theorem \ref{Thm:Dedekind} to $T(x):=\FF_1(x)$ using the prime $p$.
 Let $\tau(x)=x^{2s/p^{m}}-kx^{s/p^{m}}-1$, and suppose that
\[\overline{\tau}(x)=x^{2s/p^{m}}-\overline{k}x^{s/p^{m}}-1=\prod_{i=1}^k\overline{\tau_i}(x)^{e_i},\]
where the $\overline{\tau_i}(x)$ are irreducible.
Then $\overline{T}(x)=\prod_{i=1}^k\overline{\tau_i}(x)^{p^{m}e_i}$. Thus, we can let
\[g(x)=\prod_{i=1}^k\tau_i(x) \quad \mbox{and}\quad h(x)=\prod_{i=1}^k\tau_i(x)^{p^{m}e_i-1},\] where the $\tau_i(x)$ are monic lifts of the $\overline{\tau_i}(x)$. Note also that
\[g(x)h(x)=\prod_{i=1}^k\tau_i(x)^{e_i}=\overline{\tau}(x)+pr(x),\]
for some $r(x)\in \Z[x]$. Then, in Theorem \ref{Thm:Dedekind}, we have that
   \begin{align*}
   F(x)&=\dfrac{g(x)h(x)-T(x)}{p}\\
   &=\dfrac{(\overline{\tau}(x)+pr(x))^{p^{m}}-T(x)}{p}\\
   &=\sum_{j=1}^{p^m-1}\frac{\binom{p^m}{j}}{p}\overline{\tau}(x)^j(pr(x))^{p^m-j}+p^{p^m-1}r(x)^{p^m}+\dfrac{\overline{\tau}(x)^{p^m}-T(x)}{p},
   \end{align*}
   which implies that
   \[\overline{F}(x)=\overline{\left(\dfrac{\overline{\tau}(x)^{p^m}-T(x)}{p}\right)}.\]
 Suppose that $\overline{\tau}(\alpha)=0$. Then
 \[\overline{\tau}(\alpha)^{p^m}=\left(\beta^2-\overline{k}\beta-1\right)^{p^m}=0,\]
 where $\beta=\alpha^{s/p^{m}}$, so that $\alpha^s=\beta^{p^m}$. Since $f(\varepsilon)=f(\overline{\varepsilon})=0$, we assume, without loss of generality, that $\alpha^s=\varepsilon^{p^m}$.
Thus,
\[\overline{F}(\alpha)=-\overline{\left(\frac{T(\alpha)}{p}\right)}=-\overline{\left(\frac{\alpha^{2s}-k\alpha^s-1}{p}\right)}=-\overline{\left(\frac{\varepsilon^{2p^m}-k\varepsilon^{p^m}-1}{p}\right)}=0,\]
since $\varepsilon^{2p^m}-k\varepsilon^{p^m}-1 \equiv 0\pmod{p^2}$ by Lemma \ref{Lem:Main1}. Therefore, by Theorem \ref{Thm:Dedekind}, we conclude that $\left[\Z_K:\Z[\theta]\right]\equiv 0\pmod{p}$ and $\FF_1(x)$ is not monogenic, which completes the proof in this direction.

($\Leftarrow$)
 Assume now that no prime divisor of $s$ is a $k$-Wall-Sun-Sun prime.
   By Lemma \ref{Lem:Basic1}, we see that $\FF_0(x)=f(x)$ is monogenic. 
   Note that $\FF_0(\varepsilon)=0$.
 For $n\ge 0$, define
  \[\varepsilon_n:=\varepsilon^{1/s^n} \quad \mbox{and} \quad K_n:=\Q(\varepsilon_n).\]
   Then $\varepsilon_0=\varepsilon$ and, since $\FF_0(x)$ is monogenic, we have that $\Delta(\FF_0)=\Delta(K_0)$. Additionally, by Lemma \ref{Lem:Irreducible},
 \[\FF_n(\varepsilon_n)=0\quad \mbox{and} \quad [K_{n+1}:K_n]=s\] for all $n\ge 0$.  We assume that $\FF_n(x)$ is monogenic, so that $\Delta(\FF_n)=\Delta(K_n)$, and we proceed by induction on $n$ to show that $\FF_{n+1}(x)$ is monogenic. Let $\Z_{K_{n+1}}$ denote the ring of integers of $K_{n+1}$.  Consequently, by Theorem \ref{Thm:CD}, it follows that
 \begin{equation*}\label{Eq:CD}
 \Delta(\FF_n)^s \mbox{ divides } \Delta(K_{n+1})=\dfrac{\Delta(\FF_{n+1})}{[\Z_{K_{n+1}}:\Z[\varepsilon_{n+1}]]^2},
 \end{equation*}
 which implies that
 \[[\Z_{K_{n+1}}:\Z[\varepsilon_{n+1}]]^2 \mbox{ divides }  \dfrac{\Delta(\FF_{n+1})}{\Delta(\FF_n)^s}.\]
We see from Theorem \ref{Thm:Swan} that
\[ \abs{\Delta(\FF_n)^s}=s^{2ns^{n+1}}(k^2+4)^{s^{n+1}}\quad \mbox{ and } \quad \abs{\Delta(\FF_{n+1})}=s^{2(n+1)s^{n+1}}(k^2+4)^{s^{n+1}}.\]
 Hence,
 \[\abs{\dfrac{\Delta(\FF_{n+1})}{\Delta(\FF_n)^s}}=s^{2s^{n+1}}.\] Thus, it is enough to show that $\gcd(s,[\Z_{K_{n+1}}:\Z[\varepsilon_{n+1}]])=1$.
  Suppose then that $p$ is a prime divisor of $s$.

  If $p\mid k$, then we can apply item \eqref{JKS:I2} of Theorem  \ref{Thm:JKS} to $\FF_{n+1}(x)$.  Observe that $A_2=k/p$ and $B_1=0$, so that the first condition of item \eqref{JKS:I2} does not hold.  If
 \[A_2\left((-B)^{M_1}A_2^{N_1}-(-B_1)^{N_1}\right)=k^3/p^3 \equiv 0 \pmod{p},\] then $k\equiv 0 \pmod{p^2}$, and $p$ is a $k$-Wall-Sun-Sun prime by Lemma \ref{Lem:k not squarefree}, contradicting the fact that $s$ has no such prime divisors. Hence, the second condition of item \eqref{JKS:I2} holds and therefore, $[\Z_{K_{n+1}}:\Z[\varepsilon_{n+1}]]\not \equiv 0\pmod{p}$ in this case.

Suppose next that $p\nmid k$ and $p^m\mid \mid s$, with $m\ge 1$.
We apply Theorem \ref{Thm:Dedekind} to $T(x):=\FF_{n+1}(x)$ using the prime $p$. We assume that $\overline{\tau}(\alpha)=0$, with  \[\tau(x)=x^{2s^n/p^{mn}}-kx^{s^n/p^{mn}}-1,\] and use the same argument as in the other direction of the proof where we showed that $\FF_1(x)$ is not monogenic when $p\nmid k$. Omitting the details, we arrive at
\[\overline{F}(\alpha)=-\overline{\left(\frac{T(\alpha)}{p}\right)}=-\overline{\left(\frac{\alpha^{2s^n}-k\alpha^{s^n}-1}{p}\right)}
=-\overline{\left(\frac{\varepsilon^{2p^{mn}}-k\varepsilon^{p^{mn}}-1}{p}\right)}.\]
 Therefore, $\overline{F}(\alpha)=0$ if and only if $\varepsilon^{2p^{mn}}-k\varepsilon^{p^{mn}}-1\equiv 0 \pmod{p^2}$, which is true if and only if $p$ is a $k$-Wall-Sun-Sun prime by Lemma \ref{Lem:Main1}. Since $s$ has no prime divisors that are $k$-Wall-Sun-Sun primes, it follows that $\gcd(\overline{F},\overline{g},\overline{h})=1$. Hence, by Theorem \ref{Thm:Dedekind}, we conclude that $[\Z_{K_{n+1}}:\Z[\varepsilon_{n+1}]]\not \equiv 0\pmod{p}$ and $\FF_{n+1}(x)$ is monogenic, which completes the proof of the theorem.
 \end{proof}


\section*{Data Availability Statement}
The author confirms that all relevant data are included in the article.

\end{document}